\documentclass[oneside,11pt]{amsart}
\usepackage{amssymb, amsmath, amsthm, enumerate,graphicx}
\usepackage{hyperref, tikz}
\usepackage{thmtools, comment,color,todonotes}
\usepackage{tikz-cd}

\DeclareMathOperator{\cat}{CAT}

\newtheorem{thm}{Theorem}[section]

\newtheorem{lem}[thm]{Lemma}

\newtheorem{corollary}[thm]{Corollary}
\newtheorem*{thm*}{Theorem}

\newtheorem{innercustomthm}{Theorem}
\newenvironment{customthm}[1]
{\renewcommand\theinnercustomthm{#1}\innercustomthm}
{\endinnercustomthm}

\newenvironment{customcor}[1]
{\renewcommand\theinnercustomthm{#1}\innercustomthm}
{\endinnercustomthm}

\newtheorem{prop}[thm]{Proposition}

\theoremstyle{definition}
\newtheorem{defn}[thm]{Definition}

\newtheorem{conjecture}[thm]{Conjecture}
\newtheorem{remark}[thm]{Remark}
\newtheorem{question}[thm]{Question}
\newtheorem*{Outline}{Outline}
\newtheorem*{ack}{Acknowledgements}

\newcommand{\bnd}{\partial}

\author{Michael Ben-Zvi and Robert Kropholler}
\title{Right-angled Artin Group Boundaries}

\newcommand{\Z}{\mathbb{Z}}
\newcommand{\R}{\mathbb{R}}
\newcommand{\B}{\mathcal{B}}
\newcommand{\T}{\mathcal{T}}
\newcommand{\W}{\mathcal{W}}

\DeclareMathOperator{\Itin}{Itin}
\DeclareMathOperator{\nexus}{Nex}
\DeclareMathOperator{\nerve}{Nerve}

\begin{document}
			\begin{abstract}
			In all known examples of a CAT(0) group acting on CAT(0) spaces with non-homeomorphic CAT(0) visual boundaries, the boundaries are each not path connected. In this paper, we show this does not have to be the case by providing examples of right-angled Artin groups which exhibit non-unique CAT(0) boundaries where all of the boundaries are arbitrarily connected. We also prove a combination theorem for certain amalgams of CAT(0) groups to act on spaces with non-path connected visual boundaries. We apply this theorem to some right-angled Artin groups.
		\end{abstract}
	\maketitle

	\section{Introduction}
	
	When Gromov introduced hyperbolic groups \cite{Gromov87}, he showed that their boundaries are well-defined in the sense that if $\Gamma$ is hyperbolic and acts geometrically on spaces $X$ and $Y$, then $\bnd X$ is homeomorphic to  $\bnd Y$. He asked if the same is true for $\cat(0)$ groups. Croke and Kleiner answered this question in the negative using the fundamental group of the surface amalgam in Figure \ref{fig:torus complex} \cite{CK00}. Throughout we denote this group $CK$. Changing the angle of intersection between the curves $b$ and $c$ from $\pi/2$ to anything else changes the universal cover in a way which makes the resulting boundaries non-homeomorphic. Later, Wilson showed that for any pair of angles, the corresponding universal covers have non-homoemorphic boundary \cite{JWilson}. Thus $CK$ admits uncountably many visual boundaries. Croke and Kleiner extended their results to a larger class of groups \cite{CK02} and Mooney further generalized this work, providing even more examples of $\cat(0)$ groups with non-unique boundary \cite{Moo10}.

	\begin{figure}[h]
		\includegraphics[width=0.5\textwidth]{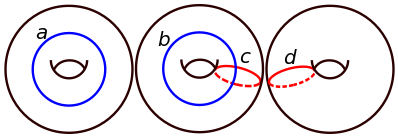}
		\centering
		\caption{Three tori with the curve $a$ identified with $b$ and the curve $c$ identified with $d$.}
		\label{fig:torus complex}
	\end{figure}

	The boundaries for every one of these groups has one thing in common: it is not path connected. This leads to the following questions
	\begin{question}
		Suppose $G$ acts geometrically on a $\cat(0)$ spaces $X$ and $Y$. 
		\begin{itemize}
			\item 	If $\bnd X$ is path connected, are $\bnd X$ and $\bnd Y$ homeomorphic? 
			\item If $\bnd X$ is $n$-connected, are $\bnd X$ and $\bnd Y$ homeomorphic? 
		\end{itemize}
	\end{question}
	In this paper, we show the answer to these questions is `no.'
	\begin{customthm}{\ref{thm: nonhomeo}}
		For each $n$, there is a group $G_n$ and $\cat(0)$ spaces $X_n$ and $Y_n$ admitting geometric group actions by $G_n$ with the following properties:
		\begin{itemize}
			\item $\bnd X_n$ and $\bnd Y_n$ are $n$-connected
			\item $\bnd X_n$ and $\bnd Y_n$ are not homeomorphic
		\end{itemize} 
	\end{customthm} We do this by considering $G_n = CK\times \Z^{n+1}$. Once again the choice of angle between the curves $b$ and $c$ produces spaces with non-homeomorphic boundaries. One should note that due to the central $\Z^{n+1}$ the boundary of any CAT(0) space upon which $G_n$ acts geometrically has boundary $S^{n}\ast Z$, for some space $Z$, and so is $n$-connected. We ask the following question. 
	
	\begin{question}
		Does there exist group $G$ which acts on two CAT(0) spaces $X, Y$ geometrically and $\bnd X$ is $n$-connected but $\bnd Y$ is not?
	\end{question}
	
	Another point of interest arising from studying $CK$ is trying to understand when the phenomenon of the boundary not being path connected occurs. For hyperbolic groups path connectivity is equivalent to one-ended \cite{BestvinaMess,BowCutPoints,Swarup}. This is also true for groups which are CAT(0) and hyperbolic relative to free Abelian subgroups \cite{Benzvi}. In this paper we give a combination theorem which gives conditions on a space $X$ together with a geometric action of $A\ast_C B$ to have path disconnected boundary. Namely, we prove the following:
	
	\begin{customthm}{\ref{thm:notpathconn}}
		Let $G = A\ast_C B$ be a CAT(0) group acting geometrically on a CAT(0) space $X$. Suppose that there is a subspace $X_C$ with a geometric action of $C$ which separates $X$. Suppose further that $A$ acts geometrically on a subspace $X_A$ satisfying the following:
		\begin{enumerate}
			\item $X_A$ has a connected block decomposition,
			\item $\bnd X_A$ is not path connected, and
			\item $\Lambda(C)\subset\nexus(X_A)$
		\end{enumerate}
		then $\bnd X$ is not path connected.
	\end{customthm}
	
	See Section \ref{section block} for relevant definitions.
	
	The last requirement is needed in light of the example $F_2\ast_\Z F_2$ giving a surface group. Thus it is not enough to assume that $\bnd X_A$ is not path connected. 
	
	Throughout the paper the example to have in mind is when $A = CK$ and $C = \langle a, d\rangle$. In this setting, $CK$ acts naturally on its Salvetti complex and this space does not have a path connected visual boundary. If we then amalgamate over the visual $F_2$ subgroup, then the resulting group acts on a space with a non-path connected visual boundary.
	
	The major application of this theorem comes in looking at right-angled Artin groups (RAAGs). Let $A_\Gamma$ be the RAAG with defining graph $\Gamma$. Then we have sufficient conditions showing when $S_\Gamma$, the universal cover of the Salvetti complex of $A_\Gamma$, has a non-path connected visual boundary. Let $P_4$ be the defining graph for $CK$ with vertex set $V(P_4)=\{a,b,c,d\}$ in the usual ordering.
	\begin{customcor}{\ref{cor:raagscor}}
		Let $A_{\Gamma}$ be a RAAG admitting a graph of groups as in Figure \ref{fig:graph} where $H_i$ is a proper parabolic subgroup of $CK$. Then $\bnd S_{\Gamma}$ is not path connected. 
	\end{customcor}
	
	One class of RAAGs which fits into the hypotheses of this corollary are those whose defining graph is an $n$-cycle for $n\geq 5$, which was the original motivation for proving Theorem \ref{thm:notpathconn}. This fits into the following conjecture of Mihalik:

	\begin{conjecture}
		$\bnd S_\Gamma$ is path connected if and only if $\Gamma$ is a join.
	\end{conjecture}
	
	We can also study the groups from \cite{Moo10}. These groups are of the form $(G_-\times \Z^n)\ast_{\Z^n}(\Z^n\times \Z^m)\ast_{\Z^m}(\Z^m\times G_+)$ where $G_-, G_+$ are infinite CAT(0) groups. As discussed above these groups are shown to have non-unique CAT(0) boundary \cite{Moo10}. For certain of these groups we can show that the boundary is not path connected. 
	
	\begin{customthm}{\ref{thm:mooney}}
		Let $G$ be of the form $(G_-\times \Z)\ast_{\Z}(\Z\times \Z)\ast_{\Z}(\Z\times G_+)$ or $(G_-\times \Z)\ast_{\Z}(\Z\times \Z^2)\ast_{\Z^2}(\Z^2\times G_+)$, where $G_-$ and $G_+$ are CAT(0) groups. Then $G$ acts on a CAT(0) $X$ and $\bnd X$ is not path connected. 
	\end{customthm}
	
	We conjecture that this result holds more generally. 
	
	\begin{conjecture}
		Let $G = (G_-\times \Z^n)\ast_{\Z^n}(\Z^n\times \Z^m)\ast_{\Z^m}(\Z^m\times G_+)$ where $G_-, G_+$ are infinite CAT(0) groups. Then $G$ act on a CAT(0) space $X$ such that $\bnd X$ is not path connected. 
	\end{conjecture}
	
\begin{Outline}
	In Section \ref{section: background}, we provide the necessary background information on $\cat(0)$ boundaries, block decompositions, and right-angled Artin groups. Section \ref{section: amalgams} is devoted to proving Theorem \ref{thm:notpathconn}. We do this by showing if a closed set $C$ separates closed sets $A$ from $B$ in $X$, then $\bnd C$ separates $\bnd A$ and $\bnd B$ in $\bnd X$. We show that the hypotheses for Theorem \ref{thm:notpathconn} force this type of separation. In Section \ref{subsection: RAAG}, we apply Theorem \ref{thm:notpathconn} to RAAGs and groups coming from \cite{Moo10}. Lastly, in Section \ref{section: nonunique}, we prove Theorem \ref{thm: nonhomeo}.
\end{Outline}

	\begin{ack}
The authors thank Mike Mihalik and Kim Ruane for bringing the question of path-connectedness of boundaries to our attention. 
	\end{ack}

	\section{Preliminaries} \label{section: background}
	
	\subsection{$\cat(0)$ spaces, groups, and boundaries}
	
	For a thorough introduction to $\cat(0)$ spaces and groups, see \cite{BH99}. Throughout this section, let $X$ be a proper, complete $\cat(0)$ space. 
	
	\begin{defn}[Visual boundary and the cone topology]
		Fix a basepoint $q\in X$. The set $\bnd_q X$ consists of geodesic rays based at $q$. Let $\overline{B}(q,r)$ be the closed ball of raduis $r$ about $q$. Let $\pi_r:X\to \overline{B}(q,r)$ be the projection map. Define the following sets in $\overline{X}=X\cup \bnd_q X$:
		
		\[ U(c,r,D) := \{ x\in \overline{X} : d(x,q)>r, d(\pi_r(x),c(r))<D\}\]
		where $c$ is a geodesic segment or ray. Sets of this form along with balls in $X$ make up a neighborhood basis for $\overline{X}$. This is called the \textit{cone topology}.
		
		Restricting the cone topology to $\bnd_q X$, the neighborhood basis can be restated as follows:
		\[U(c,r,D) = \{ c'\in \bnd_q X: d(c(r),c'(r)< D\}\]
		for any choice of $c\in \bnd_q X$, $D>0$, $r>0$.
	\end{defn}
	
	\begin{prop}[{\cite[II.8]{BH99}}]
		For any two $q,q'\in X$, $\bnd_q X$ and $\bnd_{q'} X$ are homeomorphic. Because of this, we denote $\bnd X$ to be the visual boundary with the cone topology.
	\end{prop}

\begin{defn}[$\mathcal{Z}$-set]
	A closed subset $Z$ in a compact absolute neighborhood retracts (ANR) $Y$ is a $\mathcal{Z}$-set if the one of following equivalent conditions holds:
	\begin{enumerate}
		\item For every open set $U\subset Y$, $U-Z\hookrightarrow U$ is a homotopy equivalence.
		\item For every closed $A\subset Z$, there is a homotopy $H:Y\times [0,1]\to Y$ such that $H_0$ is the identity, $H_t|A$ is the inclusion map, and $H_t(Y-A)\subset Y-Z$ for all $t>0$.  
	\end{enumerate}
\end{defn}
It is known that $\overline{X}$ is an (ANR) \cite{GuilZStructure}, $\overline{X}$ and $\bnd X$ are compact \cite{BH99}, and $\bnd X$ is a $Z$-set \cite{BestLocal} in $\overline{X}$. The second point of the definition will be used in the proof of Lemma \ref{lem:separate}.
	
		\begin{defn}[$n$-connected]
		A topological space $Z$ is $n$-connected for $n\geq 1$ if $Z$ is non-empty, path connected, and $\pi_i(Z)= 0$ for $1\leq i \leq n$. A space is $0$-connected if it is non-empty and path connected.
	\end{defn}
	
	\subsection{Blocks and Itineraries}\label{section block}
	
	\begin{defn}[{\cite[Defintion 3.1]{Moo10}}]
		A block decomposition $\B$ of a $\cat(0)$ space $X$ is a collection of closed, convex sets call \textit{blocks} such that
		\begin{enumerate}
			\item 	$X = \bigcup \limits_{B\in \B} B$,
			\item each block intersects at least two other blocks,
			\item Parity condition: every block has a $(+)$ or $(-)$ parity such that two blocks intersect only if they have opposite parity,
			\item $\epsilon$-condition: there is an $\epsilon>0$ such that two blocks intersect iff their $\epsilon$-neighborhoods intersect.
		\end{enumerate}
	\end{defn}
	 We say a block decomposition is \textit{connected} if $\bigcup\limits_{B\in \B} \bnd B$ is path connected. When $\B$ is a connected block decomposition, we shall refer to the path component of $\bnd X$ containing a block boundary (equivalently all the block boundaries) as the \textit{nexus} of $X$ and will denote it $\nexus(X)$.
	
	A \textit{wall} is a non-trivial intersection of blocks and the set of walls is denoted $\W$. A set $C$ \textit{separates} $A$ and $B$ if every path from $A$ to $B$ passes through $C$. When $X$ has a block decomposition, if $B\cap B' = W$ for some $W\in \W$, then $W$ separates $B$ and $B'$.
	
	The \textit{nerve} of a block decomposition, denoted $\nerve(\B)$, is a graph which records (non-trivial) block intersections. There is a vertex for each block and an edge if two blocks intersect. Following results from \cite{CK00} and \cite{Moo10}, we get that the nerve is always a tree.
	\begin{lem}
		$\nerve(\B)$ is a tree.
	\end{lem}
	Since $\nerve(\B)$ is a tree, we will denote it $\T_\B$ or $\T$ when $\B$ is understood. 
	
	Fix a basepoint $x_0\in X$ which is not in any wall. We say a geodesic ray $r\colon[0,\infty)\to X$ based at $x_0$ \textit{enters} a block $B$ if for some time $t>0$, $r(t)\in B$ and $r(t)\notin B'$ for any other block $B'$. Associated to each ray is a \textit{itinerary}, denoted $\Itin_{x_0}(r)$, which is the sequence of blocks $r$ enters. When the basepoint is understood we will denote this $\Itin(r)$. If $r$ represent the point $\alpha\in \bnd X$, then we will write $\Itin(\alpha)$ to denote $\Itin_{x_0}(r)$.  
	
	There are two types of itineraries: finite and infinite. A point $\alpha\in \bnd X$ has a finite itinerary if $\alpha\in \bnd B$ for some block $B$ and has an infinite itinerary otherwise. A boundary point having finite or infinite itinerary is independent of choice of basepoint (see \cite{Moo10}). Let $\Itin(\alpha)=B_1,...,B_n$ and let $v_1,...,v_n$ be the vertices in $\T$ associated to these blocks. The geodesic in $\T$ between $v_1$ and $v_n$ consists exactly of the vertices $v_1,...,v_n$. Furthermore, if $\Itin(\alpha)$ is infinite, then the associated path is the vertices of a geodesic ray in $\T$ starting at $v_1$. From this we can define a map $p\colon\bnd X\to \T\cup \bnd T$ which sends $\alpha$ to the last block in its itinerary or to the boundary point of $\bnd \T$ associated to the itinerary of $\alpha$. 
	
	We say that a geodesic ray is {\em lonely} if it is the only geodesic with its itinerary. 

	Finally we say that a subgroup is {\em finite itinerary} if every word in the subgroup has bounded itinerary. This is equivalent to saying that the subgroup stabilises a vertex of $\T$.

	\subsection{Right-angled Artin groups and Salvetti Complex}
	
	Throughout this section, let $\Gamma$ be a graph. 
	
	\begin{defn}[Right-angled Artin group]
		The \textit{right-angled Artin group} with defining graph $\Gamma$, denoted $A_\Gamma$, is the group with presentation
		
		\[A_\Gamma  = \langle V(\Gamma)| [v_i,v_j]=1 \iff \{v_i,v_j\}\in E(\Gamma)\rangle \]
		
	\end{defn}
	
	\begin{defn}
		Given a flag complex $\Gamma$ we define $X_{\Gamma}$ as follows:
		
		For each vertex $v_i$ in $\Gamma$, let $S^1_{v_i} = S^1$ be a copy of the circle cubulated with 1 vertex. For each simplex $\sigma = [v_0, \dots, v_n]$ of $\Gamma$ there is an associated torus $T_{\sigma} = S^1_{v_0}\times\dots\times S^1_{v_n}$. If $\tau <\sigma$, then there is a natural inclusion $T_{\tau}\hookrightarrow T_{\sigma}$. Now define
		
		$$X_{\Gamma} = \left.\coprod_{\sigma<\Gamma} T_{\sigma}\middle/\sim\right.$$
		where the equivalence relation $\sim$ is generated by the inclusions $T_{\tau}\hookrightarrow T_{\sigma}$. 
		
		The {\em Salvetti complex} $S_{\Gamma}$ for $A_{\Gamma}$ is the universal cover of $X_{\Gamma}$. 
	\end{defn}

\begin{remark}
	Some authors refer to $X_\Gamma$ as the Salvetti complex and $S_\Gamma$ as its universal cover. We are interested in $S_\Gamma$ so shall refer to it as the Salvetti complex.
\end{remark}
	
	\begin{prop}
		The 1-skeleton of the Salvetti complex for $A_{\Gamma}$ is isomorphic to the Cayley graph of $A_{\Gamma}$. 
	\end{prop}
	
	\begin{defn}
		Let $A_{\Gamma}$ be a RAAG we say that $H\leq A_{\Gamma}$ is a {\em parabolic} subgroup, if it is generated by a subset of the vertices of $\Gamma$. 
	\end{defn}
	
	Parabolic subgroups are convex in the sense that they stabilise a convex subcomplex of the Salvetti complex.
	
	We will study the RAAG $CK$ defined by a path with 4 vertices with the following presentation $$CK = \langle a, b, c, d\mid [a, b], [b, c], [c, d]\rangle.$$ This group has a splitting as $(F_2\times\Z)\ast_{\Z^2}(F_2\times\Z) = \langle a, b, c\rangle_{\langle b, c\rangle}\langle b, c, d\rangle$. Whenever we refer to a splitting of $CK$ throughout the paper this will be the splitting.
	
	\begin{lem}
		The Salvetti complex for $CK$ has a connected block decomposition.
	\end{lem} 
	\begin{proof}
		The blocks are copies of $T_4\times \R$ coming from the splitting $(F_2\times\Z)\ast_{\Z^2}(F_2\times\Z)$. The union of the block boundaries is connected since each block boundary is the suspension of a Cantor set and they can be ordered in a way such that each intersects one of the previous block boundaries in $S^1$. 
	\end{proof}

	\section{Boundaries of amalgamations} \label{section: amalgams}

	In this section we will study how the boundary of an amalgamated free product can be related to the boundaries of the component pieces. 
	
	%
	
	\begin{lem}\label{lem:separate}
		Let $X$ be a $\cat(0)$ space and let $A$ and $B$ be closed subsets of $X$. Suppose there exists a closed subset $C$ such that any geodesic from $A$ to $B$ passes through $C$. Then any path in the boundary between $\Lambda(A)$ and $\Lambda(B)$ passes through $\Lambda(C)$. 
	\end{lem}
	
	\begin{proof}
		Fix a basepoint $x_0\in X$ and let $\gamma:[0,1]\to \bnd X$ be a path starting at $\alpha\in \bnd A$ and ending at $\beta\in \bnd B$. By the $\mathcal{Z}$-set property of $\overline{X}$, there is a homotopy $H:[0,1]\times [0,1]\to \overline{X}$ of $\gamma$ such that $H(t,1)=\gamma(t)$ and $H(t,s)\subset X$ for all $s\neq 1$. Furthermore, $H(0,s)$ is along the geodesic from $x_0$ to $\alpha$ and $H(1,s)$ is along the geodesic from $x_0$ to $\beta$. Since $C$ separates $A$ and $B$, for sufficiently large $s$, any path from $H(0,s)$ to $H(1,s)$ passes through $C$. This gives a sequence of points in $C$ which tend to infinity. Since $C\cup \Lambda(C)$ is compact, this sequence has a subsequence which converges in $\Lambda(C)$. By the construction, this point is along $\gamma$. 
	\end{proof}

	Using the above Lemma we can deduce our main theorem. 

	\begin{thm}\label{thm:notpathconn}
		Let $G = A\ast_C B$ be a CAT(0) group acting geometrically on a CAT(0) space $X$. Suppose $A$ and $C$ act geometrically on subspaces $X_A$ and $X_C$, respectively. Furthermore, suppose $X_C$ and its translates separate $X_A$ from the rest of $X$. Lastly, suppose $X_A$ satisfies the following:
		\begin{enumerate}
			\item $X_A$ has a connected block decomposition,
			\item $\bnd X_A$ is not path connected, and
			\item $\Lambda(C)\subset\nexus(X_A)$
		\end{enumerate}
		then $\bnd X$ is not path connected.
	\end{thm}
	\begin{proof}
		Since $\nexus(X_A)$ is path connected and $\bnd X_A$ is not path connected, we can find a point $p$ in $\nexus(X_A)$ and a point $q$ in $\bnd X_A$ which cannot be connected by a path.
		
		Suppose that we can find a path $\gamma\colon I\to \bnd X$ from $p$ to $q$. This path cannot be contained in $\bnd X_A$ as such $\gamma$ passes through $\Lambda(C)$. This gives paths from $p$ and $q$ to $\Lambda(C)$. However $\Lambda(C)$ is contained in $\nexus(X)$. Since $\nexus(X)$ is path connected, we can then obtain a path from $p$ to $q$ in $\bnd X_A$ giving a contradiction. Thus there is no path from $p$ to $q$ in $\bnd X$ completing the proof. 
	\end{proof}
	
	\subsection{Applications to right-angled Artin groups} \label{subsection: RAAG}
	
	We now move onto looking at some key examples coming from right angled Artin groups. 
	
	\begin{prop}\label{prop:amalgoverfiniteitin}
		Let $A_{\Gamma}$ be a RAAG such that $A_{\Gamma} = CK\ast_C B$ where $C$ is a finite itinerary parabolic subgroup. Then $\partial S_{\Gamma}$ is not path connected. 
	\end{prop}
	\begin{proof}
		Since $C$ is a parabolic subgroup of $A_{\Gamma}$ it acts on a convex subspace of $S_{\Gamma}$ it also separated the space $A_{\Gamma}$. The Salvetti complex for $CK$ has a connected block decomposition. This comes from the splitting of the form $(F_2\times \Z)\ast_{\Z^2}(F_2\times \Z)$. The subspaces are copies of $T_4\times \R$ and the walls are $\R^2$. The block decomposition is connected since each block boundary is the suspension of a Cantor set and these block boundaries intersect along an $S^1$.
		
		Finally since $C$ is a finite itinerary parabolic every ray in $C$ ends in the boundary of a block and as such $\Lambda(C)$ is contained in the nexus of the block decomposition. We can now apply Theorem \ref{thm:notpathconn}.
	\end{proof}
	
	We obtain the following corollary immediately:
	\begin{corollary}
		Let $\Gamma$ be a path with at least 4 vertices. Then $\bnd S_{\Gamma}$ is not path connected. 
	\end{corollary}
	
	We can also do the case that the parabolics are not finite itinerary. 
	
	\begin{lem}\label{lem:parabolicsarelonely}
		Let $H$ be a proper parabolic subgroup of $CK$ with an infinite itinerary ray $\alpha$. Let $\beta$ be an infinite itinerary ray in $CK$. Suppose $\beta$ is not in $H$. Then $\alpha$ and $\beta$ have different itineraries.  
	\end{lem}
	\begin{proof}
		There are 14 proper parabolic subgroups of $CK$. The only proper parabolic subgroups that contain infinite itinerary rays are $\langle a, b, d\rangle, \langle a, c, d\rangle$ and $ \langle a, d\rangle$. By symmetry, we just have to check the cases of $\langle a, b, d\rangle$ and $\langle a, d\rangle$. 
		
		Let us first consider the case $H = \langle a, b, d\rangle$. By moving any $b$ or $d$ to the left we can assume that $\beta$ has the form $w_1 c^n a w_2$. Where $w_1$ is a word in $H$. We see that the itinerary of $\beta$ contains the vertex $v = w_1c^na\langle b, c, d\rangle$. We will show that this vertex is not on the itinerary of any ray in $H$. 
		
		Suppose that $\alpha$ is a ray with $v$ in its itinerary. Then there is a finite prefix $w$ of $\alpha$ such that $w\langle b, c, d\rangle = v\langle b, c, d\rangle$. Thus $w^{-1}w_1c^na\in \langle b, c, d\rangle$. However by the normal forms theorem for right angled Artin groups the $a$ cannot be cancelled. Thus $v\langle b, c, d\rangle$ cannot be on the itinerary of $\alpha$. 
		
		We now consider the case of $H = \langle a, d\rangle$. Since $\beta$ is not in $H$ $\beta$ must contain at least one $b$ or at least one $c$. We will study the case the $\beta$ contains a $b$, the other case is the same. Since $\beta$ is an infinite itinerary ray it contains infinitely many occurences of $d$, as such there is at least one after the $b$. Thus $\beta$ has the form $w_1bw_2dw_3$ for words $w_i$. Thus on the itinerary of this element is the vertex $w_1bw_2d\langle a, b, c\rangle$. We will show that this vertex cannot be on the itinerary of $\alpha$. 
		
		Suppose that $w_1bw_2d\langle a, b, c\rangle$ is on the itinerary of $\alpha$. This gives an element of $h$ such that $h\langle a, b, c\rangle = w_1bw_2d\langle a, b, c\rangle$ or $h^{-1}w_1bw_2d\in\langle a, b, c\rangle$. Once again we see from the normal forms theorem for right angled Artin groups that this cannot be the case as there is no way to cancel the $d$ to the right of the $b$. 
	\end{proof}
	
	\begin{lem}
		Let $\alpha$ be an infinite itinerary ray contained in a proper parabolic subgroup $H$. Then $\alpha$ is lonely. 
	\end{lem}
	\begin{proof}
		We have already seen that $\alpha$ is separated from any ray not contained in $H$. However $H$ is CAT(0) with isolated flats and the induced splitting from the splitting of $CK$ respects the peripheral structure. Thus we can apply \cite[Lemma 10.17]{Benzvi} to see that all infinite itinerary rays are lonely.
	\end{proof}

 We can restate Proposition 6.22 from \cite{Benzvi} as follows:
	
	\begin{lem}\label{lem:containedinthenexus}
		Let $\alpha$ be an infinite itinerary ray contained in a proper parabolic subgroup. Then $\alpha$ is connected to the nexus of the block decomposition of $CK$. 
	\end{lem}

	\begin{thm}\label{thm:otherparaspli}
		Let $A_{\Gamma}$ be a RAAG which splits as $CK\ast_H B$ where $H$ is a proper parabolic subgroup. Then $\partial S_{\Gamma}$ is not path connected. 
	\end{thm}
	\begin{proof}
		From the proof of Proposition \ref{prop:amalgoverfiniteitin} we can see that we have a subspace $X_{CK}$ which has a connected block deomcposition and boundary which is not path connected. Lemma \ref{lem:containedinthenexus} shows us that $\Lambda(H)\subset \nexus(X_{CK})$ and so we can apply Theorem \ref{thm:notpathconn}. 
	\end{proof}
	
	\begin{corollary}
		Let $\Gamma$ be a graph isomorphic to a circle with at least five vertices. Then $\partial S_{\Gamma}$ is path disconnected. 
	\end{corollary}
	\begin{proof}
		Take any subpath of the circle of length 4. The two outside vertices of the path give a splitting as $A_{\Gamma} = CK\ast_{F_2} B$ where $F_2$ is the parabolic given by the two outside vertices. This is a proper parabolic subgroup of $CK$ so we can apply Theorem \ref{thm:otherparaspli}. 
	\end{proof}

	It is worth noting that the idea of the previous proofs is that we start with two points which are not connected by a path in $\partial CK$ and show that they are still not connected by a path in $\partial G$. Thus we can repeat the construction as long as we know that the two points in the $\partial CK$ subgroup are not connected by a path in the previous stage. For instance we get the following:
	
	\begin{corollary}\label{cor:raagscor}
		Let $A_{\Gamma}$ be a RAAG admitting a graph of groups as in Figure \ref{fig:graph} where $H_i$ is a proper parabolic subgroup of $CK$. Then $\bnd S_{\Gamma}$ is not path connected. 
	\end{corollary}

\begin{figure}
	\begin{tikzpicture}[scale=2.5]
	\draw 	{(1, 0) node (CK) [draw,circle,fill=black,minimum size=4pt, inner sep=0pt,label=below:$CK$] {}
		(0, 0) node (B1) [draw,circle,fill=black,minimum size=4pt, inner sep=0pt,label=left:$B_1$] {}
		(0.55, 0.7) node (B2) [draw,circle,fill=black,minimum size=4pt, inner sep=0pt,label=above:$B_2$] {}	
		(1.45, 0.7) node (B3) [draw,circle,fill=black,minimum size=4pt, inner sep=0pt,label=above:$B_3$] {}
		(2, 0) node (Bn) [label=right:$B_n$] {}
		(CK) edge node[below] {$H_1$} (B1)
		(CK) edge node[left] {$H_2$} (B2)
		(CK) edge node[right] {$H_3$} (B3)
		(CK) edge node[below] {$H_n$} (Bn)
		(1.55, 0.3) node [draw,circle,fill=black,minimum size=1pt, inner sep=0pt] {}
		(1.6, 0.25) node [draw,circle,fill=black,minimum size=1pt, inner sep=0pt] {}
		(1.65, 0.2) node [draw,circle,fill=black,minimum size=1pt, inner sep=0pt] {}
	};
	\end{tikzpicture}
	\caption{A graph of groups where each $B_i$ is a RAAG and $H_i$ is a proper parabolic subgroup of $CK$. }
	\label{fig:graph}
\end{figure}
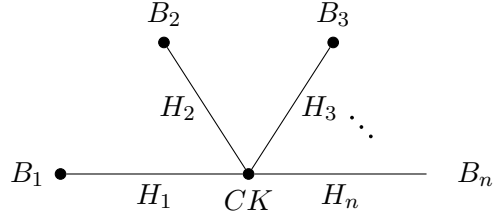
	
	In Figure \ref{fig:notappl}, we draw attention to two graphs which are not joins and our methods cannot be applied. 
	
	We also study the groups from \cite{Moo10}. These groups are of the form $(G_-\times \Z^n)\ast_{\Z^n}(\Z^n\times \Z^m)\ast_{\Z^m}(\Z^m\times G_+)$ where $G_-, G_+$ are infinite CAT(0) groups. As discussed above these groups are shown to have non-unique CAT(0) boundary \cite{Moo10}. We obtain the following theorem about path connectedness.
	
	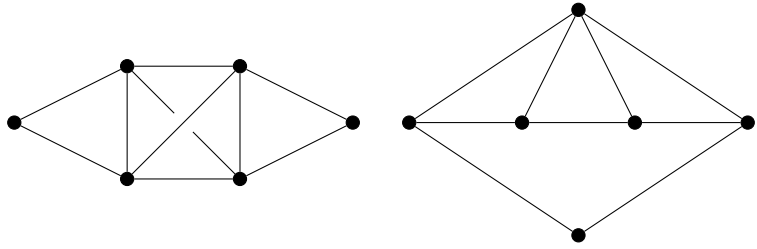
\begin{figure}
	\begin{tikzpicture}[scale =1.5]
	\tikzstyle{every node}=[draw,circle,fill=black,minimum size=5pt,inner sep=0pt]
	\draw 	{(0, 0) node (1) {}
		(1, 0.5) node (2) {}
		(1, -0.5) node (3) {}
		(2, 0.5) node (4) {}
		(2, -0.5) node (5) {}
		(3, 0) node (6) {}
		(2) edge (5)
		(3) edge[white, line width=10pt] (4)
		(1) edge (2)
		(1) edge (3)
		(2) edge (3)
		(2) edge (4)
		(3) edge (5)
		(4) edge (5)
		(4) edge (6)
		(5) edge (6)
	};
	\draw (3) -- (4);
	\begin{scope}[shift={(3.5, 0)}]
	\draw 	{(0, 0) node (1) {}
		(1, 0) node (2) {}
		(2, 0) node (3) {}
		(3, 0) node (4) {}
		(1.5, 1) node (5) {}
		(1.5, -1) node (6) {}
		(1) edge (5)
		(2) edge (5)
		(3) edge (5)
		(4) edge (5)
		(1) edge (6)
		(4) edge (6)
		(1) edge (2)
		(2) edge (3)
		(3) edge (4)
	};
	
	\end{scope}
	\end{tikzpicture}
		\caption{Two graphs for which we cannot apply our theorem and for which path connectedness of $\bnd S_{\Gamma}$ is unknown. }
		\label{fig:notappl}
	\end{figure}  
	
	\begin{thm}\label{thm:mooney}
		Let $G$ be of the form $(G_-\times \Z)\ast_{\Z}(\Z\times \Z)\ast_{\Z}(\Z\times G_+)$ or $(G_-\times \Z)\ast_{\Z}(\Z\times \Z^2)\ast_{\Z^2}(\Z^2\times G_+)$, where $G_-$ and $G_+$ are CAT(0) groups. Then $G$ acts on a CAT(0) $X$ and $\bnd X$ is not path connected. 
	\end{thm} 
	\begin{proof}
		Let $a$ be an infinite order element of $G_-$ and $d$ be an infinite order element of $G_+$. Such elements exist by work of Swenson \cite{SwensonCutPoint}.  
		
		We obtain a splitting of $(G_-\times \Z)\ast_{\Z}(\Z\times \Z)\ast_{\Z}(\Z\times G_+)$ as $$(G_-\times \Z)\ast_{\Z^2}(CK)\ast_{\Z^2}(\Z\times G_+)$$ identifying $\langle a\rangle \times \Z$ with $\langle a, b\rangle$ and identifying $\Z\times\langle d\rangle$ with $\langle c, d\rangle$. We can build a CAT(0) space $X$ on which this acts using the Equivariant Gluing Theorem \cite[II.11.18]{BH99}. As a subspace of this we have the Salvetti complex for $CK$ and the amalgamating subgroups are finite itinerary and so the hypothesis of Theorem \ref{thm:notpathconn} are satisfied. 
		
		For $(G_-\times \Z)\ast_{\Z}(\Z\times \Z^2)\ast_{\Z^2}(\Z^2\times G_+)$ we use the splitting $$(G_-\times \Z)\ast_{\Z^2}(CK)\ast_{\Z^2\ast\Z}(\Z\times (G_+\ast\Z)).$$ The proof from the first case follows verbatim for the second. 
	\end{proof}
	
	\section{Non-unique boundaries} \label{section: nonunique}
	
	In this section, we prove Theorem \ref{thm: nonhomeo}, which we restate for convenience.
	
	\begin{thm}\label{thm: nonhomeo}
		For each $n$, there is a group $G_n$ and $\cat(0)$ spaces $X_n$ and $Y_n$ admitting geometric group actions by $G_n$ with the following properties:
		\begin{itemize}
			\item $\bnd X_n$ and $\bnd Y_n$ are $n$-connected
			\item $\bnd X_n$ and $\bnd Y_n$ are not homeomorphic
		\end{itemize} 
	\end{thm}

	It is known by \cite{CK00} that the boundary of a CAT(0) group is not well defined. However, none of the known boundaries of this group are path connected. In this section we give examples of CAT(0) groups with non-unique boundary in which all boundaries are path connected. Moreover, we give similar results where all CAT(0) boundaries are $n$-connected for arbitrary $n$. 
	
	We follow closely the proof given in \cite{CK00}. Let $X = X_{\alpha}$ be the universal cover of the Leeb complex where the geodesics in the central block meet at an angle of $\alpha$. This space is a classifying space for $CK$. The groups that we will study are of the form $G_n = CK\times\Z^{n+1}$. This group acts properly cocompactly on $Z = X\times\R^{n+1}$. The boundary of $X\times\R^{n+1}$ is $\partial X\ast S^n$. We will show that we can still differentiate these spaces after taking a topological join with $S^n$. Note that this boundary is $n$-connected as $\partial X$ is not empty and $S^n$ is $n-1$ connected. 
	
	The group $G_n$ has a splitting as $F_2\times\Z^{n+2}\ast_{\Z^{n+3}}F_2\times \Z^{n+2}$. The space $Z$ has a connected block decomposition where each block is isometric to $T_4\times \R^{n+2}$ and the walls are each isometric to $\R^{n+3}$.	Since each block $B$ is isometric to the product of a tree and $\R^{n+2}$ we see that $\partial B = \mathcal{C}\ast S^{n+1}$ where $\mathcal{C}$ is a Cantor set. A \textit{pole} of $B$ is any point of $S^{n+1}$.

	\begin{lem}\cite[Lemma 3]{CK00}
		
		If $B_1, B_2$ are blocks, then one of the following holds:
		\begin{enumerate}
			\item $\partial B_1\cap \partial B_2 = \emptyset$.
			\item $B_1\cap B_2$ is a wall $W$ and $\partial B_1\cap \partial B_2 = \partial W$. 
			\item There is a block $B$ such that $B\cap B_i = W_i$ and $\partial B_1\cap\partial B_2$ is the set of poles of $B$. 
		\end{enumerate}
	\end{lem}
	\begin{proof}
		The blocks in our decomposition are a product of the blocks from \cite{CK00} and $\R^{n+1}$. 
	\end{proof}
	
	\begin{lem}\label{lem:nbhdsinjoins}
		Let $z$ be a point of $X\ast Y = X\times Y\times[0, 1]/\sim$. Suppose $z = (x, y, t)$. Then a neighbourhood $N(z)$ of $z$ in $X\ast Y$ is homeomorphic to $$N(z) = 
		\begin{cases}
		N(x)\times C(Y) & \text{if } t = 0,\\
		N(y)\times C(X) & \text{if } t = 1,\\
		N(x)\times N(y)\times (t-\epsilon, t+\epsilon) & \text{if } 0<t<1.
		\end{cases}$$ Where $C(X)$ is the open cone on $X$. 
	\end{lem}
	\begin{proof}
		The preimage of each the above sets under the quotient map is open thus they are open sets in the quotient. 
	\end{proof}
	
	\begin{lem}\cite[Lemma 4]{CK00}
		
		Suppose that $\lambda\in \partial B$  and $\lambda$ is not the pole of any block other than $B$. Then the path component of $\lambda$ in a suitable neighbourhood $\Lambda$ of $\lambda$ is contained in $\partial B$. 
	\end{lem}
	\begin{proof}
		Since the space in question is $\partial X\ast S^n$ we can understand neighbourhoods by \ref{lem:nbhdsinjoins}. Since $\lambda$ is not a pole of any other block we must be in the case $t\neq 1$. Thus the neighbourhood is homeomorphic to $N(x)\times D^{n+1}$ and the proof from \cite{CK00} follows. 
	\end{proof}
	
	We say that $\lambda$ is a {\em vertex} if there is a neighbourhood $U$ of $\lambda$ such that the path component of $\lambda$ in $U$ is homeomorphic to $D^{n}\times C(\mathcal{C})$. By the previous lemma, poles which do not come from $S^n$ are vertices. 
	
	We say that a path is {\em safe} if it only goes through finitely many vertices and no points with a neighbourhood homeomorphic to $D^n\times C(\partial X)$. We say a path component is {\em safe} if there is a safe path between any pair of points in the path component. With this definition of safe path and vertex we recover the results of \cite[Section 7]{CK00}. 
	
	\begin{lem}\label{lem:unionofblocksissafe}\cite[Lemma 6]{CK00}
		
		The union of the block boundaries without $S^n$ is a safe path component. 
	\end{lem}
	
	\begin{lem}\cite[Lemma 7]{CK00}
		
		Let $c\colon [0, 1]\to \partial Z$ be a path and suppose that $c(0)$ has an infinite itinerary. Then either $c(t)$ has the same itinerary for all $t$ or there is a point with finite itinerary. 
	\end{lem}
	\begin{proof}
		If the path avoids the $S^n$ in the decomposition $\partial Z = \partial X\ast S^n$, then the proof follows from \cite[Lemma 7]{CK00}. On the other hand every point on $S^n$ has finite itinerary. 
	\end{proof}
	
	\begin{corollary}\cite[Corollary 8]{CK00}
		
		There is a unique safe component of $\partial Z$ which is dense, namely the space describe in Lemma \ref{lem:unionofblocksissafe}. 
	\end{corollary}
	
	We say that a disk $D^{n+1}\subset \cup_{B}\partial B$ is a {\em simplex} if its of the form $I\ast S^n$ for an edge $I$ defined in \cite[Section 8]{CK00}. As in \cite{CK00} we  see that simplices are contained in the boundary of a single block. Moreover, the boundary of a simplex is $S^0\ast S^n$ and every point not in $S^n$ is a vertex and no interior point of $D^{n+1}$ is a vertex. We can see that all the boundary points not in $S^n$ are either the poles of a single block or $D^{n+1}$ is contained in $\partial W$ for some wall $W$. We call an arc $I\subset \cup_{B}\partial B$ a {\em special edge} if every point is a vertex. Let $x$, $y$ be vertices not on a special edge. Then they are in the same block if they are on the boundary of more than one simplex. They are in adjacent blocks if they are on the boundary of a unique simplex. Either both poles are on a special edge. A subset of $\cup_{B}\partial B$ is a block boundary if it is the union of all simplices intersecting it.
	
	A {\em hemisphere} is a subset of $\partial B$ of the form $S^n\ast I$ where $I$ is a longitude from \cite{CK00}. Let $B$ be a block and $\mathcal{P}$ be the set of poles in adjacent blocks. Let $H$ be a hemisphere in $B$. Then by \cite[Lemma 9]{CK00}, we have that $H\smallsetminus\bar{\mathcal{P}}$ has 3 components if $\alpha \neq \frac{\pi}{2}$ and 2 components if $\alpha = \frac{\pi}{2}$. 
	
	Finally, following \cite[Section 10]{CK00} we can distinguish these boundaries as any homeomorphism will take safe path components to safe path components, block boundaries to block boundaries, poles to poles and hemispheres to hemispheres. Then the previous paragraph gives a contradiction, concluding the proof of Theorem \ref{thm: nonhomeo}.
	
	At this time there is no known example of a group one of whose boundaries is path connected and the other not. Moreover, there is no known example of a CAT(0) group admitting two non homeomorphic boundaries one of which is locally connected. Our examples do nothing to answer these questions.

	\bibliographystyle{alpha}
	\bibliography{bibliography}
\end{document}